\documentclass[a4paper,12pt]{amsart}

\usepackage[latin1]{inputenc}
\usepackage[T1]{fontenc}
\usepackage[english]{babel}

\usepackage{mathrsfs}
\usepackage{amscd}
\usepackage{amsfonts}
\usepackage{amsmath}
\usepackage{amssymb}
\usepackage{amstext}
\usepackage{amsthm}
\usepackage{amsbsy}

\usepackage{xspace}
\usepackage[all]{xy}
\usepackage{graphicx}
\usepackage{url}
\usepackage{latexsym}

\makeatletter
\newcommand*{\rom}[1]{\expandafter\@slowromancap\romannumeral #1@}
\makeatother

\theoremstyle{definition}

\newtheorem{fact}{fact}

\newtheorem{thm}[fact]{Theorem}
\newtheorem{lemma}[fact]{Lemma}
\newtheorem{prop}[fact]{Proposition}

\newtheorem{defini}[fact]{Definition}

\title{Towards a Church-Turing-Thesis for Infinitary Computations}
\author{Merlin Carl}

\begin{document}
\maketitle

\begin{abstract}
We consider the question whether there is an infinitary analogue of the Church-Turing-thesis. To this end, we argue that there is an intuitive notion of transfinite computability and build a canonical model, called Idealized Agent Machines ($IAM$s) 
of this which will turn out to be equivalent in strength to the Ordinal Turing Machines defined by P. Koepke.
\end{abstract}

\section{Introduction}

Since \cite{ITTM}, various generalizations of classical notions of computability to the transfinite have been given and studied. The Infinite Time Turing Machines ($ITTM$s) of Hamkins and Lewis generalized classical Turing machines to transfinite working time.
Ordinal Turing Machines ($OTM$s) (see \cite{OTM})  and Ordinal Register Machines ($ORM$s) further generalized this by allowing working space of ordinal size. 
Recently, a transfinite version of $\lambda$-calculus was introduced
and studied \cite{Sey}. It was soon noted (see e.g. \cite{Fi}) that the corresponding notion of computability enjoys a certain stability under changes of the machine model: For example, the sets of ordinals computable by $OTM$s and $ORM$s both coincide with
the constructible sets of ordinals.

A similar phenomenon is known from the models of classical computability: Turing machines, register machines, recursive functions, $\lambda$ calculus etc. all lead to the same class of computable functions. In the classical case, this is taken as evidence for
what is known as the Church-Turing-Thesis ($CTT$), i.e. the claim that these functions are exactly those computable in the `intuitive sense' by a human being following a rule without providing original input. This thesis plays an important role in mathematics:
It underlies, for example, the - to our knowledge undisputed\footnote{It has been remarked that there are challenges to the claim that no physical device could decide such questions, see e.g. \cite{Hog} and \cite{NeGe}. However, here we are interested in the capabilites of idealized computing agents. Whether what such devices do can be considered to be a computation in the intuitive sense rather than the observation of an incomputable process is a question we won't consider here.} - view that Matiyasevich's theorem \cite{Ma} settles Hilbert's $10$th problem or that Turing's work \cite{Tu1} settles the Entscheidungsproblem. The study of recursive functions gets a lot of its attraction from this well-grounded belief that they coincide with this intuitive notion of computability.

It therefore seems natural to ask whether something similar can be said about transfinite models of computation, i.e. whether these models are mere `ordinalizations' of the classical models or whether they actually `model' something, whether there is an intutive
concept of transfinite computability that is captured by these models: Hence, we ask for an infinitary Church-Turing-thesis ($ICTT$).\\
There seems to be some evidence that a satisfying $ICTT$ should be obtainable. Beside the stability of the corresponding notion of computability mentioned above, it also became common to describe and communicate the activity of such machines in rather 
informal terms: Rather than writing an actual program for e.g. deciding number-theoretical statements with an $ITTM$, it generally suffices to explain that the machine will e.g. `search through the naturals for a witness'. It usually soon becomes clear to someone
with a basic familiarity with these models that such a method can indeed be implemented and will lead to the right results. Indeed, we will usually find such a `process description' much easier to grasp than an actual implementation. 
This indicates that we indeed possess an intuitive understanding of what these machines can do which is based on an understanding of infinite processes rather than the formal definition of the machine. We aim at connecting infinitary models of 
computation with a natural notion. Here, `natural' means that the notion can be obtained and described independently from the models and that it is in some sense present in normal (mathematical) thinking. Such a notion should furthermore serve as
a background thesis explaining the equivalence of the different models, should (in analogy with the classical Church-Turing-thesis) justify the use of informal 'process descriptions' to prove the existence of formally specified programs and, ideally,
allow mathematically fruitful applications, similar to the role the classical $CTT$ plays in e.g. Hilbert's $10$th problem.\\

In this work, we offer evidence for the claim that notions of transfinite computation are indeed naturally present in mathematical (and possibly in everyday) thinking and that these notions are captured by the transfinite 
machine models we mentioned.\footnote{To be precise, we will argue for this claim in the case of $OTM$s and $ORM$s. Whether similar approaches are available for other models as well is briefly adressed at the end of this paper.}
This will allow us to formulate an $ICTT$.

This article is structured as follows: We begin by describing an approach of mathematical philosophy initiated by P. Kitcher \cite{Ki}, where mathematical objects are modelled as mental constructions of idealized agents. We also indicate that such idealizations
are indeed present in understanding mathematics. After that, we work towards a formal notion of a computing transfinite agent, obtaining the notion of an Idealized Agent Machine ($IAM$). Then, we show that the computational power of an $IAM$ coincides
with that of $OTM$s and $ORM$s (which we will summarize under the term `standard models' from now on). Finally, we state (a candidate for) an $ICTT$ and discuss whether it meets the above requirements.

\section{Idealized Constructions and Idealized Agents in Mathematics}

In this section, we briefly describe the view on the philosophy of mathematics described in [Kitcher].  We use his account as a demonstration that the concept
of transfinite agents can be motivated and has arisen completely independent from our considerations. Furthermore, we want to indicate how these views can be fruitful for infinitary computations (and vice versa) and bring them into interaction.
Finally, his work serves us as a first introduction to the notion of idealized agents. We will then demonstrate that this notion seems indeed to be present in 
mathematical language and understanding.

\subsection{Kitcher's idealized-agents-view of mathematics}

In a nutshell, Kitcher attempts to justify an empiricist account of mathematics by describing mathematics as an idealization of operations with real-world
objects like grouping them together, adding an object to a pile of objects etc. These actions in themselves already are a kind of primitive mathematics, limited by our practical constraints. What is usually called mathematics is obtained by forming a theory 
of idealized operations in a similar way that, say, a theory of idealized gases is formed: We abstract away from certain `complicating factors' like e.g. our factual incabability of indefinitely adding objects to a collection. 
Mathematics is then the study of idealized operations, or, equivalently, of the operations of idealized agents.

Upon reading this, one might wonder how this account is supposed to make sense of the large parts of mathematics which, like axiomatic set theory, deal with actual infinite objects. Kitcher's reply to this is simply that this is a mere question
of the degree of idealization:

\begin{quote}
\cite{Ki}, p. $146$: I see no bar to the supposition that the sequence of stages at which sets are formed is highly superdenumerable, that each of the stages corresponds to an instant in the life of the constructive subject, and that the subject's activity
is carried out in a medium \textit{analogous} to time, but far richter than time. (Call it `supertime'.) ... The view of the ideal subject as an idealization of ourselves does not lapse when we release the subject from the constraints of our time.
\end{quote}

Comparing Kitcher's account of axiomatic set theory with his treatment of arithmetic or intuitionistic mathematics, mathematical areas can roughly be characterized by the degree of idealization, i.e. by considering how remote the underlying operations
are from our actual capabilities. The agent working in `supertime' mentioned in the quote above seems to belong to a benchmark of idealization. As this is the degree of idealization corresponding to set theory in Kitcher's account, we will refer to
it as the `idealized agent of set theory' from now on.\footnote{Similar ideas are mentioned in other accounts on the philosophy of mathematics. For example, in \cite{Wa}, S. 182, we find the following: `The overviewing of an infinite range of objects presupposes an infinite intuition which is an idealization. 
Strictly speaking, we can only run through finite ranges (and perhaps ones of rather limited size only).'}

Not unexpectedly, several issues with this approach can and have been raised: E.g. about the ontological status of these idealized agents (discussed in \cite{Ho}), whether this degree of idealization still admits an explanation of the applicability
of mathematics, whether and how certain large cardinals can be accomodated in this account etc. Nevertheless, the imagination of an idealized agent or an idealized mental activity seems 
to be in the background of large parts of mathematical understanding in one way or the other.
%We think like this and formulate like this. 
In fact, there are numerous common figures of speech in mathematical textbooks and even more in spoken conversation that point to such (implicit) notions: For example, in many proofs of the Bolzano-Weierstra{\ss}-theorem, `we' 
are supposed to `pick' a number from a subintervall containing infintely many elements of a given sequence. One might find this problematic: In a naive sense, of course, we cannot do this, as in general, 
we will not know which intervall that is.\footnote{This is the reason why Bolzano-Weierstrass is intuitionistically invalid.} However, this problem doesn't seem to come up in understanding this proof. In fact, agent-based formulations
generally seem to increase understanding and make constructions more imaginable rather than leading into conflicts with our factual limitations.
A similar observation holds for e.g. proofs of the well-ordering principle from the axiom of choice, and in general for
many uses of transfinite recursion or transfinite induction. Another example would be the various places in mathematical logic where constructions are explained by interpreting them as transfinite `games' between two `players'.

\subsection{Degrees of idealization and the Church-Turing-Thesis}

In the Church-Turing-Thesis, recursiveness is stated to capture the intuitive meaning of `computable'. 
However, if the intuitive meaning of `computable' is taken as `possible for a human being working without understanding', then literally, this is of course false: 
What we can actually do is very limited: In general, a recursive function is far away from being computable by `a man provided with paper, pencil, and rubber, and subject to strict discipline' (\cite{Tu}).
But this fact is quite irrelevant for e.g. Hilbert's $10$th problem, which asks for a `finite' procedure, not a practical one. In the $CTT$, we are hence in fact facing a notion of an idealized computing subject.

Usually, this idealization goes from certain factual bounds to `arbitarily large, but finite'. But there seems to be a distinguished intuitive notion of computability going beyond this: For example, there is little to no trouble with 
the idea of testing all even numbers for being a sum of at most two primes. In fact, this thought experiment seems to be at least part of the reason the Goldbach conjecture is generally assumed to have a definite truth value.
On the other hand, no such intuition supports the idea of e.g. searching through $V$ looking for a bijection between $\mathbb{R}$ and $\aleph_{1}$, not even if one assumes $CH$ to have a definite truth value.\footnote{Searching through $L$ or its stages, 
on the other hand, seems again quite reasonable, as $L$ is canonically well-ordered.} 
The idea of a transfinite systematic procedure for obtaining certain objects or answering certain questions hence allows for a clear distinction: Not every formulation that at the surface looks like a `process description' is eligible as an indication of a 
computation of an idealized agent. Our goal is to find an exact characterization of those procedures that are.

\section{A model for idealized Agents}

%Why not take $ITTM$s? -> Because they seem unjustified. We will show that they can be re-obtained by an appropriate analysis.
Even if one accepts that, beyond finiteness, clear degrees of idealization of our activity can be concretely captured, the standard models are not as canonical a model of it as e.g. Turing machines are in the finite case. In the one direction,
it does indeed seem plausible that the actions of an $OTM$ are available to a transfinite idealized agent and that hence everything computable by an $OTM$ should be computable by such an agent: The aspects of an $OTM$-computation going beyond
classical computability consist in elementary limit operations like forming the limes inferior of a sequence of $0$s and $1$s. But the other direction is not as clear: 
For example, the limit rule of $OTM$s seems to be rather arbitrary. The intuition here is that other reasonable choices of limit rules will not change the class of computable objects, but it is exactly the intuition leading there 
that we want to capture here. We see no direct path from idealized agents to the standard models known so far. Our approach is hence to develop a formal notion of a transfinitely computing agent modelled after our intuition 
and then see how it relates to the standard models. It turns out that it does indeed describe the same notion of computability, which we consider a good piece of evidence for our thesis.

The notion we are about to develop will be called Idealized Agent Machines ($IAM$s). $IAM$s are meant to give a very liberal account of the computational activity of idealized agents. In fact, one might get the impression that what we model as 
a single step of an $IAM$ is really a series of lengthy sub-computations and that we are hence far to generous in attributing abilities to our idealized agent. However, we will demonstrate that even this liberal notion is equivalent to the standard models.
Therefore, we don't need to claim that $IAM$s are a very accurate model for the intuition of transfinite computations: we only need it to be strong enough to include that intuition. We can then argue that if such an intuition is really present - as we tried
to show above, then it is grasped by the standard models, as, in the end, we will arrive at the following implications:
\begin{center}
 $OTM$-computable\\$\underset{(1)}{\implies}$ computable by the idealized agent of set theory\\$\underset{(2)}{\implies} IAM$-computable\\$\underset{(3)}{\implies} OTM$-computable
\end{center}
Here, implication $(3)$, being a claim about two notions expressable in the language of set theory, is provable (in $ZFC$) and implication $(1)$ is very natural (see above). 
It is step $(2)$ that depends on the plausibility of the analysis and modelling we are about to give.

An ideal computing agent works as follows: 
%He has a finite set of instructions. 
At each time, he has a complete memory of his earlier computational activity. Also, he has a working memory where he may store information. We assume that the working memory consists of separate `places', each containing one symbol from a finite alphabet.\footnote{The finiteness of the alphabet could in fact be dropped without changing the class of computable functions we ultimately obtain. However, we consider this a reasonable assumption for the notion we are about to model and hence decided against
taking the effort to demonstrate this.}

The agent is working in according with instructions that determine his activity. Certainly, any kind of operation that can be considered an idealization of an activity we are actually capable of must be describable by finite means.
We hence stipulate that the instructions are given by some finite expressions. Based on the instructions, it must be possible at each time to determine what to do 
(e.g. which new symbols to write) on the basis of the computational activity so far. 

We propose to model this in the following way: There should be a first-order formula $\phi(x,y)$ such that, if the computational activity so far is given by $c$ and $p$ is a place in the memory,  $\phi(c,p,s)$ holds iff $s$ is 
the the symbol that should be written in place $p$ after $c$.
Here, it must be possible to evaluate $\phi$ by mere inspection of $c$. Even if `inspection' may be taken in an idealized sense here as well, this should certainly mean that the appearing quantifiers should in some sense be `bounded' by $c$. 
We will make this precise below.\footnote{The choice of first-order logic might be objectional; we feel that e.g. second-order logic would be inappropriate, for it would require the agent to have access to an external notion of
set which is not determined from his computational activity. However, we are certainly interested in plausible alternatives and whether they would turn out to lead to an equivalent notion of computability.}\\
This description does not depend on any assumptions on the structure of time. It is hence sufficiently general to yield a notion of transfinite computability once an appropriate notion of transfinite time is introduced.

\subsection{Supertime and Superspace}

In the passage quoted in the first paragraph, Kitcher suggests that set theory can be considered as the outcome of the mental activity of an idealized agent working in `a medium analogous to time, but far richer than time'. 
Here, we want to argue that the only sensible choice for such a medium are ordinals. 
 In his argumentation, it is also implicitely assumed that the agent not only has a non-standard working time, but also the ability to `store' the outcome of his work, e.g., infinite memory or at least infinite writing space. We will argue that it
is natural and harmless to assume that the writing space of an idealized agent is indexed by ordinals.

Certainly, we intend a notion of time as a medium of a deterministic computation to be a linear ordering. But we can say more.
The computational activity has to start at some point. Every other state may depend on this earlier state and hence has to take place at a moment after the starting point. Hence, the `medium of computation' has to have a unique minimal element.

Whenever the agent has carried out a certain amount of computational activity, he has to know what to do next, i.e. there must be a unique next state for him to assume. This next state has to take place at some point of time. Hence, the medium in which he
computes has to contain a unique next element after those through which the activity passed so far. Put differently: For every initial segment of time, there has to be a unique time point preceeded by all moments in the 
initial segment and only by those.
This leads to the following notion of `supertime': A `supertime' is a linearly ordered set\footnote{The outcome might be different if one would allow `class time'. We don't pursue this further here.} 
$(X,\leq)$ with a unique minimal element $\mu$ and such that, for every proper initial segment $I$ of $X$, there is a $\leq$-minimal $x_{I}\in X$ such that
$\forall{t\in I}t<x_{I}$. It is now easy to see that this means that all candidates for supertime are (isomorphic to) ordinals:

\begin{prop}
 Let $(X,\leq)$ be a linearly ordered set such that, for every $I\subsetneq X$ which is downwards closed (i.e. $x<y\in I$ implies $x\in I$), there is a minimal $x_{I}\in X$ such that $\forall{t\in I}t<x_{I}$. Then $(X,\leq)$ is isomorphic to an ordinal.
\end{prop}
\begin{proof} 
Note that $\emptyset$ is downwards closed in $(X,\leq)$ and let $\mu:=x_{\emptyset}$. Obviously, $\mu$ is the unique minimal element of $X$.\\
Let $A\subseteq X$. Consider the set $Y:=\{x\in X|x<A\}$. It is easy to see that $Y$ is an initial segment of $X$. We claim that $x_{Y}$ is a minimal element of $A$.\\
To see that $x_{Y}\in A$, assume otherwise. As every element smaller than $x_{Y}$ is in $Y$ and hence smaller than every element of $A$, it follows that $x_{Y}<A$.
But this implies $x_{Y}\in Y$, so $x_{Y}<x_{Y}$, a contradiction. So $x_{Y}\in A$ and every $z<x_{Y}$ satisfies $z\notin A$. Thus $x_{Y}$ is indeed a minimal element of $A$. As $\leq$ is linear, $x_{Y}$ is unique with this property.\\
This implies that $(X,\leq)$ is a well-ordered set. Hence, it is isomorphic to an ordinal.
\end{proof}

However, not all ordinals are suitable as such a medium: For example, if our medium allows two procedures to be carried out, it should also allow to carry out one after the other. Also, it should be possible to have a procedure as a `subroutine' of another to 
be repeatedly called by the other. Finally, the class of ordinals itself provides an attractive unification of appropriate computation times; hence we allow computations carried out without fixing a particular ordinal in advance.\\
Appropriate candidates for supertime hence turn out to be ordinals which are closed under ordinal addition and multiplication and $On$ itself. In the following, we will - for the sake of simplicity -  focus on the broadest case where the underlying time is $On$. Note that this notion of 
supertime matches well with the way transfinite constructions are commonly communicated and imagined: It is completely normal to relate stages of such a construction by expressions coming 
from the relation of time points and state that e.g. `earlier on, we made sure that'. In fact, it is hard to talk about transfinite constructions avoiding such expressions.

We imagine our agent to be equipped with a sufficient supply of place for writing symbols. We assume that this space is partioned into
slots and that each slot is uniquely recognizable. There is a canonical well-ordering on the set of used slots: Namely, each slot is at 
some point of time used for the first time. Via this property, this slot is henceforth identifiable. We may hence assume for our convenience that the slots are indexed with ordinals from the very beginning: 
That is, the working memory is at any time a function from some ordinal
$\alpha$ into the set $S$ of symbols.\footnote{This point could be strengthened by modelling space in a more general way and then proving the resulting notion to be equivalent with the one obtained here. 
However, this requires a cumbersome analysis and the gain in plausibility seems to be too limited to justify it.}

Finally, even if we allow - as we will - several symbols to be re-written in one step, an adequate model of computing time and space should also impose some bounds on the space that can be actually used after computing for $\tau$ many steps.
We model this intuition by the extra condition that, at time $\tau$, only slots with index in $\tau$ may contain written symbols.\footnote{This condition may seem to be too strict compared to the overall very liberal model we set up. However, this choice 
is technically the least cumbersome; furthermore, we conjecture from our experience so far that every bound that is reasonably explicit in $\tau$ will ultimately lead to the same class of computable functions.}

\subsection{Idealized Agent Machines}

We will now describe a formal model for the concept developed above. The instructions will be given by a first-order statement in an appropriate language, which can be evaluated on the basis of an initial segment of a computation.

We let $L_{c}$ be the first-order language with equality, a binary function symbol $C(x,y)$ and a binary relation symbol $\leq$. The intended meaning of $C(x,y)=z$ is that, at time $x$, $z$ is the symbol in the $y$th place, while $\leq$ is the
ordering relation of ordinals.\\ 
If $A$ is a finite set (the alphabet) and $\tau$ an ordinal, then a $\tau$-state for $A$ is a function $f:\alpha\rightarrow A$, where $\alpha\leq\tau$. We denote the class of $\tau$-states for $A$ by $S_{A}^{\tau}$.\\
%A will be the symbols our machine uses. 
A function $F$ with $dom(F)=:\tau\in On$ and $F(\iota)\in S_{A}^{\iota}$ for all $\iota<\tau$ is called an $A$-$\tau$-precomputation. For $F$ an $A$-$\tau$-precomputation, an $L_{c}$-formula $\phi$, $\vec{s}\in A^{<\omega}$, $\vec{\alpha}\in(\tau+1)^{<\omega}$, we 
define $[\phi(\vec{\alpha},\vec{s})]_{\tau}^{F}$, the truth value of $\phi(\vec{\alpha},\vec{s})$ in $F$, recursively (omitting the parameters where possible): $[C(\alpha,\beta)=x]_{\tau}^{F}=1$ if $\alpha<\beta$ or $F(\alpha)(\beta)=x$, otherwise $[C(\alpha,\beta)=x]_{\tau}^{F}=0$; $[x\leq y]_{\tau}^{F}=1$ iff
$x,y\in On$ and $x\leq y$, otherwise $[x\leq y]_{\tau}^{F}=0$; $[x=y]_{\tau}^{F}=1$ iff $x=y$, otherwise $[x=y]_{\tau}^{F}=0$; $[\neg\phi]_{\tau}^{F}=1-[\phi]_{\tau}^{F}$; $[\phi\wedge\psi]_{\tau}^{F}=[\phi]_{\tau}^{F}[\psi]_{\tau}^{F}$; 
and $[\exists{x}\phi(x)]_{\tau}^{F}=1$ iff there is $\iota\in\tau$ such that $[\phi(\iota)]_{\tau}^{F}=1$, otherwise $[\exists{x}\phi(x)]_{\tau}^{F}=0$.\\
An $L_c$-formula $\phi(x,y,z)$ is an $IAM$-program iff, for all $\tau\in On$, $\alpha\leq\tau$ and all $A$-$\tau$-precomputations $F$, there is exactly one $s\in A$ such that $[\phi(\tau,\alpha,s)]_{\tau}^{F}=1$.
%It allows our idealized agent to determine his next actions from the activity carried out so far. This is very similar to a set-theoretical recursion; the difference
%lies in the restriction of the quantifiers to the past computation.\\
If $\phi$ is an $IAM$-program, $A$ a finite set, $\tau\in On$ and $F$ an $A$-$\tau$-precomputation, then we define $\mathbb{S}_{\phi,\tau,F}:\tau\rightarrow A$, the state of the $IAM$-computation with $\phi$ at time $\tau$ after $F$, 
by letting $\mathbb{S}_{\phi,\tau,F}(\alpha)$ be the unique $s\in A$ such that $\phi(F,\alpha,s)$ holds for $\alpha<\tau$.\\
Furthermore, we define $\mathbb{I}_{\phi}^{\tau}$, the $\tau$-th initial segment of the $IAM$-computation with $\phi$ at time $\tau$, recursively by letting $\mathbb{I}_{\phi}^{0}:=\emptyset$,
$\mathbb{I}_{\phi}^{\tau+1}:=\{(\tau,\mathbb{S}_{\phi,\tau,\mathbb{I}_{\phi}^{\tau}})\}\cup\mathbb{I}_{\phi}^{\tau}$
and $\mathbb{I}_{\phi}^{\lambda}:=\bigcup_{\iota<\lambda}\mathbb{I}_{\phi}^{\iota}$ for $\lambda$ a limit ordinal.\\
So far, our machines have no notion of halting. We therefore assume that all our $IAM$s have a special symbol $\mathbb{H}$ in their alphabet. 
The $IAM$-computation by $\phi$ is said to have stopped at time $\tau$ iff $\mathbb{I}_{\phi,\tau}(\tau)(0)=\mathbb{H}$, i.e. if the first symbol in the memory at time $\tau$ is $\mathbb{H}$.\\
An $IAM$-computation by $\phi$ will hence start with an empty tape and then repeatedly apply the $\mathbb{S}$-operator to obtain the next state, taking unions at limits.\\
It is easy to see from the boundedness of the formula evaluated in each step that this notion of computability is absolute insofar $IAM$-computations are absolute between transitive models of $ZFC$. We can also account for computations with a non-empty input 
and computations with parameters in these terms by adjusting the initial memory content. %In particular, we can use this to emulate calculations with finitely many ordinal parameters, as shown in the following definition.

\begin{defini} $X\subseteq On$ is $IAM$-computable iff there exists an $IAM$-program $\phi$ such that, for every $\alpha\in On$, there is $\tau\in On$ such that, if $\chi_{\alpha}$ is the characteristic function of
$\alpha$ in $On$ and $F=(0,\chi_{\alpha})$, we have $\mathbb{S}_{\phi,\tau,F}(0)=\mathbb{H}$ and $\mathbb{S}_{\phi,\tau,F}(1)=1$ iff $\alpha\in On$.\\
Similarly, $f:On\rightarrow On$ is $IAM$-computable iff there is an $IAM$-program $\phi$ such that, for every $\alpha\in On$, there is $\tau\in On$ such that $\mathbb{S}_{\phi,\tau,F}(0)=\mathbb{H}$, $\mathbb{S}_{\phi,\tau,F}(f(\alpha)+1)=1$ and
$\mathbb{S}_{\phi,\tau,F}(\iota)=0$ for $\iota\notin\{0,f(\alpha)+1\}$, where again $F=(0,\chi_{\alpha})$ and $\chi_{\alpha}$ is the characteristic function of $\alpha$ in $On$.\\
We say that a set $X\subseteq On$ or a function $f:On\rightarrow On$ is $IAM$-computable from finitely many ordinal parameters iff there exists a finite set $p\subset On$, an $IAM$-program $\phi$ using the alphabet $A$ and an $a\in A$
such that $\phi$ computes $X$ (or $f$, respectively) when the following change is made for all $\tau<\alpha\in On$ in the definition of the $\alpha$-th state $\mathbb{S}_{\phi,\alpha,\mathbb{I}_{\phi}^{\alpha}}$: If $\beta\in p$, then
$\mathbb{S}_{\phi,\alpha,\mathbb{I}_{\phi}^{\alpha}}(\beta)$ is set to $a$.
\end{defini}

\section{Idealized Agent Machines, ordinal computability and the $ICTT$}

Having developed our formal model for infinitary computations, it is now rather straightforward to show that, in terms of computability, it is equivalent to the standard models. As the elobarate versions are quite long and cumbersome, 
we merely sketch the arguments here. \\
\begin{lemma}{\label{element}}
(a) There is an $L_{c}$-formula $\phi_{lim}$ such that, for any precomputation $F$ with $dom(F)=\tau$, we have $[\phi]_{\tau}^{F}=1$ iff $\tau$ is a limit ordinal. Furthermore, the statement $\alpha=\beta+1$ is expressable by an $L_{c}$-formulas
$succ(\alpha,\beta)$.\\
(b) Let $A\subset\omega$ be finite. There is an $L_{c}$-formula $\phi_{liminf}(x,y)$ such that, for any $\tau\in On$, $a\in On$, $b\in A$ and any $A$-$\tau$-precomputation $F$, $[\phi_{liminf}(a,b)]_{\tau}^{F}$ holds iff 
$b=\liminf{((F(\iota))(a))_{\iota<\tau}}$. \\
(c) Let $P$ be an $OTM$-program, and let $\sigma=(i,\alpha,t)$ be a triple coding a state in the computation with $P$, where $i$ is codes the current state of the program, $\alpha$ the head position and
$t:\tau\rightarrow\{0,1\}$ the tape content. There are $L_{c}$-formulas $\phi^{P}_{state}(i,\alpha,t,j)$, $\phi^{P}_{head}(i,\alpha,t,\beta)$ and $\phi^{P}_{tape}(i,\alpha,t,s)$ such, for any pre-computation $F$ with $dom(F)=\gamma+1$,
$[\phi_{state}(i,\alpha,t,j)]_{\gamma+1}^{F}=1$, $[\phi^{P}_{head}(i,\alpha,t,\beta)]_{\gamma+1}^{F}=1$ and $[\phi^{P}_{tape}(i,\alpha,t,s)_{\gamma+1}^{F}=1$ hold iff applying $P$ in the state $\sigma$ leads into
the new state $(j,\beta,t^{\prime})$, where $t^{\prime}:\tau+1\rightarrow\{0,1\}$ is given by $t^{\prime}(\alpha)=s$ and $t^{\prime}(\zeta)=t^{\prime}(\zeta)$ for $\zeta\neq\alpha$.
\end{lemma}
\begin{proof}
(a) Take $\phi_{lim}$ to be $\forall{x}\exists{y}(x\leq y\wedge\neg(x=y))$. First assume that $\tau$ is a\\ limit ordinal. Then $[\phi_{lim}]_{\tau}^{F}=1-[\exists{x}\forall{y}(\neg(x\leq y)\vee x=y))]_{\tau}^{F}$. Now\\
$[\exists{x}\forall{y}(\neg(x\leq y)\vee x=y))]_{\tau}^{F}=1$ iff there exists $x\in\tau$ with $[\forall{y}(\neg(x\leq y)\vee x=y)]=1$, which is equivalent to 
$[\neg\exists{y}(x\leq y\wedge x\neq y)]_{\tau}^{F}=1\leftrightarrow[\exists{y}(x\leq y\wedge x\neq y)]_{\tau}^{F}=0$, which means that there is no $y<\tau$ such that $[x\leq y\wedge x\neq y]_{\tau}^{F}=1$, i.e. such that
$x\leq y\wedge x\neq y$ holds. But such an $x$ obviously cannot exist if $\tau$ is a limit ordinal. The other direction works in the same way, again by simply unfolding the definition of the truth predicate. The second statement is 
similarly immediate.

(b) As $A=\{a_1,...,a_n\}$ is finite, we can define $\leq$ on $A$ by taking $a<b$ to be $\bigvee_{a_{i}\leq b}a_{i}=a$. Now take $\phi(a,b)$ to be $\exists{x}\forall{z}(x\leq z\implies b\leq C(z,a)\wedge\forall{x}\exists{z}(x\leq z\wedge C(z,a)=b)$.

(c) The required formulas are immediate from $P$ and the fact that limit ordinals are $L_{c}$-definable. To give an example, if $P$ requires to change from state $i$ to state $j_{1}$ when the symbol under the reading head (at position $\alpha$) 
is currently $\iota_{1}$ and to state $j_{2}$ when the symbol is $\iota_{2}$, we can express this through the $L_{c}$-formula \\
$\phi_{i}(\alpha,j)\equiv \exists{\gamma}(((\neg\exists{\beta}succ(\gamma,\beta)\wedge((C(\gamma,\alpha)=\wedge j=j_{1})\vee(C(\gamma,\alpha)=\wedge j=j_{2})))$. 
%Here $\beta$ is the last state in the pre-computation $F$.
\end{proof}

\begin{thm}
 Let $f:On\rightarrow On$ be $OTM$-computable. Then $f$ is $IAM$-computable.
\end{thm}
\begin{proof}
Let $P$ be an $OTM$-program for computing $f$. Suppose wlog that $P$ uses $s\geq 3$ many states and put $A:=\{0,1,...,s\}$. We will represent states of the $OTM$-computation as sequences $(a_i|i\in\alpha)$ where $a_0\in\{1,2,...,s\}$ codes the 
inner state of the machine and the $a_{\iota}$ code the tape content. Let $b_{i}=a_{i+1}$ for $i\in\omega$ and $b_{\iota}=a_{\iota}$ otherwise. 
To express the head position, we put $b_{\iota}=2$ if the $\iota$th cell of the Turing tape contains a $0$ and the head is currently at position $\iota$, $b_{\iota}=3$ if the $\iota$th tape content is $1$ and the head is currently at 
position $\iota$; otherwise, the $b_{\iota}$ will just agree with the tape content.\\
Using the last lemma, one can now construct an $L_{c}$-formula $\phi$ such that $\mathbb{I}_{\phi}^{\alpha}$ represents the state and tape content of $P$ at time $\alpha$ in the way we described.
\end{proof}

\begin{thm}
 Let $x\subset On$ be a set of ordinals. Then $x$ is $IAM$-computable from a finite set of ordinals iff it is $OTM$-computable from a finite set of ordinals.
\end{thm}
\begin{proof}
By \cite{Koe}, $x\subseteq On$ is $OTM$-computable from finitely many ordinal parameters iff $x\in L$. But it is not hard to see by adapting the theorem above that $OTM$-computations in finitely many parameters can be simulated by an $IAM$ that
hence every $OTM$-computable $x$ is also $IAM$-computable. 
On the other hand, as $IAM$-computations are definable in $L$, every $x$ $IAM$-computable from finitely many ordinal parameters must be an element of $L$. Hence the classes of $IAM$-computable sets of ordinals and $OTM$-computable sets
of ordinals both coincide with the constructible sets of ordinals and hence with each other.
\end{proof}

\begin{thm}
 $f:On\rightarrow On$ is $IAM$-computable iff it is computable by an ordinal Turing machine ($OTM$) without parameters.
\end{thm}
\begin{proof}(Sketch)
We saw above that an $OTM$ can be simulated on an $IAM$.\\
For the other direction, we indicate how to simulate an $IAM$ by an $OTM$. Let a finite $A$ and an $IAM$-program $\phi$ be given.\footnote{Note that a variant 
of an $OTM$ working with finitely many symbols $\sigma_1,...,\sigma_n$ can be simulated by an $OTM$ using 
only $0$ and $1$ by representing $s_i$ as $\underbrace{0...0}_{n-i}\underbrace{1...1}_{i}$.} 
To see how to emulate one computation step, assume we have safed the sequence $\textbf{s}:=(s_{\iota}|\iota<\tau)$ 
of $IAM$-states up to $IAM$-computing time $\tau$ so far on an extra tape $T_1$, separated by an extra symbol. The techniques from \cite{OTM} for evaluating the bounded truth predicate can then be adapted to compute $s_{\tau}$ on a second tape, using a third
tape as a scratch tape. For this, we compute, for each $\alpha\leq\tau$, $[\phi(\tau,\alpha,s)]_{\tau}^{\textbf{s}}$ for each $s\in A$ until we find the unique $\bar{s}$ with $[\phi(\tau,\alpha,\bar{s})]_{\tau}^{\textbf{s}}=1$, so that $s_{\tau}(\alpha)=\bar{s}$. 
Finally, we copy $s_{\tau}$ to the end of $T_1$ to obtain a representation of $(s_{\iota}|\iota<\tau+1)$.
\end{proof}

This shows, up to our analysis in section $3$ and the restriction to working time and space $On$, that the intuitive concept of transfinite computability coincides with $OTM$-computability.
Hence, we can finally close this section by stating our candidate for an $ICTT$:

\begin{center}
 \textbf{Infinitary Church-Turing-Thesis}: A function $f:On\rightarrow On$ is computable by the idealized agent of set theory following a deterministic rule iff it is computable by an $OTM$.
\end{center}

\section{Conclusion and further Work}

We have argued that there is an intuitive notion of transfinite computability and that rendering it precisely leads us to a notion of transfinite computability equivalent with
$ORM$- and $OTM$-computability. Consequently, the constructible hierarchy was obtained as the realm of this idealized activity. 
This suggests that these models indeed capture some general intuitive concept and hence that results about these models can be interpreted
as results about this notion. Accordingly, one should expect interesting applications to general mathematics: For example, one might consider measuring the complexity of an object
or a function by the computational ressources necessary to compute it. This would give a precise meaning to the question whether certain objects granted to exist by indirect proofs can be 
`concretely constructed', even if this construction is allowed to be transfinite. In particular, it suggests connections of transfinite computability to reverse mathematics as exhibited in \cite{KoeWe}.

However, our argument has the drawback of being model-dependent: We develop a certain notion of computability from the informal idea of an idealized agent, hopefully along plausible lines. It would be preferable to
have a formal notion of transfinite computation not refering to a particular model; this could be obtained by an appropriate axiomatization of transfinite computations similar to approaches that have been made in the classical case.
(See e.g. \cite{DeGu}. See also \cite{KoeSy}.)

Another question is whether a similar approach will work for other models like e.g. $ITTM$s. This is likely to be more difficult, as our coarse approach of approximating the activity of an idealized agent is not available here:
As it is shown in \cite{FrWe}, there are natural alternative choices for the limit rules that lead to larger classes of computable functions.

\end{document}